\documentclass[12pt,english]{amsart}
\usepackage[utf8]{inputenc}
\usepackage{a4wide}
\usepackage{amsfonts}
\usepackage{amsthm}
\usepackage{color}
\usepackage{amsmath,amssymb}
\usepackage{graphicx}
\usepackage{babel}
\usepackage{color}

\newcommand{\R}{\mathbb{R}}  
\newcommand{\BB}{\mathbb{B}}    

\newcommand{\N}{\mathbb{N}} 
\renewcommand{\S}{\mathbb{S}}
  
\newcommand{\dysk}{D}   
\newcommand{\ann}{A}
\newcommand{\uu}{\mathcal{U}}
\newcommand{\vv}{\mathcal{V}}

\def\w{\widetilde}
\newcommand{\ve}{\varepsilon}
\newcommand{\eps}{\varepsilon}
\newcommand{\vp}{\varphi}
\newcommand{\h}{\mathcal{H}}

\newcommand{\brac}[1]{\left({#1}\right)}
\newcommand{\bracb}[1]{\bigl({#1}\bigr)}

\newcommand{\norm}[1]{\left\|{#1}\right\|}

\newcommand{\arccot}{\mathop{\mathrm{arc\, cot}}}

\newtheorem{thm}{Theorem}[section]
\newtheorem{lem}[thm]{Lemma}
\newtheorem{cor}[thm]{Corollary}

\newtheorem{defi}[thm]{Definition}
\newtheorem{rem}[thm]{Remark}

\numberwithin{equation}{section}
\newcommand{\invisible}[1]{}

\title[The Lavrentiev gap for harmonic maps into spheres]{The Lavrentiev gap phenomenon \\for harmonic maps into spheres 
\\ holds on a dense set of zero degree boundary data}
\author{Katarzyna Mazowiecka \&\ Paweł Strzelecki}
\date{\today}

\begin{document}

	\begin{abstract}
		
		\baselineskip 15pt
		
	We prove that for each positive integer $N$ the set of smooth, zero 
	degree maps $\psi\colon\S^2\to \S^2$ which have the following 
	three properties: 
	\begin{enumerate}
	\item[(i)] there is a unique minimizing harmonic map 
		$u\colon \BB^3\to \S^2$ which agrees with $\psi$ on 
		the boundary of the unit ball;
	\item[(ii)] this map $u$ has at least $N$ singular points in $\BB^3$;
	\item[(iii)] the Lavrentiev gap phenomenon holds for $\psi$, i.\,e., 
		the infimum of the Dirichlet energies $E(w)$ of all smooth 
		extensions $w\colon \BB^3\to\S^2$ of $\psi$ is strictly larger 
		than the Dirichlet energy  $\int_{\BB^3} |\nabla u|^2$ 
		of the (irregular) minimizer $u$,
	\end{enumerate}
	is dense in the set $\mathcal{S}$ of all smooth zero degree maps 
	$\phi\colon \S^2\to\S^2$ endowed with the $W^{1,p}$--topology, 
	where $1\le p<2$. This result is sharp: it fails in the 
	$W^{1,2}$--topology of $\mathcal{S}$. 
\end{abstract}

\maketitle

\parskip 3pt

\section{Introduction}

In this note, we revisit a well-known topic, the study of singularities of 
maps $u\colon \BB^3\to\S^2$ which minimize the Dirichlet integral
\begin{equation}
\label{Dirichlet}
E(u)=\int_{\BB^3}|\nabla u|^2 dx\, , \qquad u\in W^{1,2}(\BB^3,\S^2)
\end{equation}
under a prescribed boundary condition $u\!\mid_{\partial\BB^3}\, =\varphi\colon \S^2\to\S^2$. Here, $\BB^3$ stands for the open unit ball in $\R^3$, $\S^2$ is the unit sphere, and
\[
W^{1,2}(\BB^3,\S^2)=\{v=(v_1,v_2,v_3)\in W^{1,2}(\BB^3,\R^3)\colon |v(x)|=1 \text{ for a.e. }x\in \BB^3\}\, .
\]
Moreover, for a map $\varphi$ in the fractional Sobolev space $H^{1/2}(\S^2,\S^2)$ we write
\[
W^{1,2}_\varphi(\BB^3,\S^2)=\{v\in W^{1,2}(\BB^3,\S^2)\colon 
v\!\mid_{\partial\BB^3}=\varphi \text{ in the trace sense}\}\, .
\]
Minimizers of the Dirichlet integral \eqref{Dirichlet} in $W^{1,2}_\varphi(\BB^3,\S^2)$ satisfy the Euler--Lagrange system
\begin{equation}
\label{H}
\left\{
\begin{array}{rcl}
-\Delta u & = & |\nabla u|^2 u\qquad\mbox{in $\BB^3$,}\\
u\!\mid_{\partial\BB^3}\, & = &\varphi\, .
\end{array}
\right.
\end{equation}
The main motivation behind the present work was to reach a deeper understanding of the mechanisms governing the onset of singularities of solutions, and the cardinality and structure of the set of minimizing solutions for a~fixed boundary condition. We also wanted to know whether the \emph{Lavrentiev gap phenomenon}, cf.\ \eqref{HLgap} below, occurs typically (in a precise topological meaning).  Despite the work of numerous experts over the last three decades, this topic is still not fully understood. Our main result states, roughly speaking, that even in the case when there is no purely topological reason for the solution of \eqref{H} to be discontinuous, singularities of $u$ \emph{do} occur under arbitrarily small (in the $W^{1,p}$ sense, for $1\le p<2$) perturbations of an \emph{arbitrary} smooth boundary data $\varphi$.

Before giving formal statements of the results, let us sketch a broader perspective.

When $\deg\varphi\not=0$, all solutions of \eqref{H} in $W^{1,2}_\varphi(\BB^3,\S^2)$ obviously have singularities, as $\varphi$ has no continuous extension $u\colon \BB^3\to \S^2$. By a celebrated classic theorem of Schoen and Uhlenbeck \cite{SU1} the singular set of a \emph{minimizing} solution of \eqref{H} consists of isolated points. By another theorem of Almgren and Lieb \cite{AL}, if the boundary condition $\varphi$ has square integrable derivatives on $\S^2$, then the number of these points does not exceed a constant multiple of the \emph{boundary energy} $\int_{\S^2} |\nabla_T \varphi|^2 d\sigma$. (Non-minimizing solutions can behave in a wild way: Rivi\`{e}re \cite{Riv} proves that for any non-constant boundary data $\varphi$ there exists an everywhere discontinuous solution of the harmonic map system \eqref{H}.)

However, even when $\varphi\colon \S^2\to\S^2$ satisfies $\deg\varphi=0$ --- so that a priori there is no topological obstruction for a map $u\in W^{1,2}_\varphi(\BB^3,\S^2)$ to be continuous --- minimizers of $E$ in $W^{1,2}_\varphi(\BB^3,\S^2)$ might be singular because this is energetically preferable. Hardt and Lin~\cite{HL} give an example of a smooth zero degree boundary data $\widetilde\varphi\colon \S^2\to\S^2$ which is $H^{1/2}$-close to a constant map and has the following properties:
\begin{enumerate}
	\renewcommand{\labelenumi}{{\rm (\alph{enumi})}}
	\item Each minimizer $v$ of $E$ in $W^{1,2}_{\widetilde\varphi}(\BB^3,\S^2)$ has at least $N$ singular points (the number $N$ can be prescribed a priori);
	\item The Lavrentiev gap phenomenon holds for $E$ in $W^{1,2}_{\widetilde\varphi}(\BB^3,\S^2)$. By this, we mean the following inequality:
	\begin{equation}
		\label{HLgap}
		\mu(\widetilde\varphi) := \min_{W^{1,2}_{\widetilde\varphi}(\BB^3,\S^2)} E(u) < \overline\mu(\widetilde\varphi) :=
		\inf_{W^{1,2}_{\widetilde\varphi}(\BB^3,\S^2)\cap C^0(\overline{\BB}^3)} E(u).
	\end{equation}
\end{enumerate}
An immediate consequence of \eqref{HLgap} is that $W^{1,2}_{\widetilde\varphi}(\BB^3,\S^2)\cap C^0(\overline{\BB}^3)$ is not dense in $W^{1,2}_{\widetilde\varphi}(\BB^3,\S^2)$.

As Bethuel, Brezis and Coron have shown, cf. \cite[Theorem~5]{BBC}, for boundary conditions $\varphi$ of zero degree, the Lavrentiev gap phenomenon is equivalent to the fact that all minimizing harmonic maps in $W^{1,2}_{\varphi}(\BB^3,\S^2)$ have singularities. 
%If this is the case, then one can also infer that there are infinitely many different solutions of the harmonic map system \eqref{H} for the given boundary condition $\varphi$, cf. \cite[Theorem~6]{BBC} (see also Pakzad \cite[Theorem~1]{Pakzad}). 
Other examples of unexpected and counterintuitive behavior of singularities of minimizing harmonic maps have been given by Almgren and Lieb in \cite{AL}. In particular, a minimizer $u$ of $E$ in $W^{1,2}_{\varphi}(\BB^3,\S^2)$ can have a large number of singular points even if $\det\nabla_T \varphi\equiv 0$ on $\S^2$ and $\varphi$ maps the whole sphere $\S^2$ to a smooth curve $\gamma$. The abstract of \cite{AL} ends with the phrase: ``in particular, singularities in $u$ can be unstable under small perturbations of $\varphi$.''\footnote{We have just changed Almgren and Lieb's notation from $\varphi,\psi$ to our $u$, $\varphi$ respectively.}

Our main result ascertains that the message of the last sentence, \emph{singularities can be unstable}, may be strengthened, i.\,e., replaced with a firm \emph{singularities \emph{are} unstable}, at least when one takes into account small perturbations of the boundary data in the topology of each of the space $W^{1,p}$, $1\le p<2$. Here is the precise statement.

\begin{thm} 
\label{main:thm}	
	Assume that $\varphi\in C^\infty(\S^2,\S^2)$ is an arbitrary smooth map with $\deg \varphi =0 $ and $1\le p<2$. Then, for each $\eps>0$ and each $N\in \N$ there exists a map $\widetilde\varphi\in C^\infty(\S^2,\S^2)$ such that
	\begin{enumerate}
	\renewcommand{\labelenumi}{{\rm (\roman{enumi})}}
		\item $\deg\widetilde\varphi=0$;
		\item $\|\varphi-\widetilde\varphi\|_{W^{1,p}} < \eps$ and ${\mathcal H}^2\big(\{x\in \S^2\colon \varphi(x)\not= \widetilde\varphi(x)\}\big)< \eps$;
		\item the Dirichlet integral $E$ has precisely one minimizer $\widetilde u\in W^{1,2}_{\widetilde\varphi}(\BB^3,\S^2)$; moreover, $\widetilde u$ has at least $N$ point singularities in $\BB^3$.
	\end{enumerate}
\end{thm}

Combining the above result with Bethuel, Brezis and Coron, \cite[Theorems 5--6]{BBC}, one immediately obtains the following.

\begin{cor} Assume that $\varphi\in C^\infty(\S^2,\S^2)$ and $\deg \varphi =0 $. Let $\widetilde\varphi\in C^\infty(\S^2,\S^2)$ be given by Theorem~\ref{main:thm}. Then the Lavrentiev gap phenomenon \eqref{HLgap} holds for $\widetilde\varphi$.
%and the harmonic map system \eqref{H} with $\varphi$ replaced by $\widetilde\varphi$ has infinitely many solutions.
\end{cor}

It is a natural question whether the occurrence of such boundary data is a \emph{typical property} in the class of all maps of degree zero, i.\,e., whether the set of mappings $\widetilde\varphi\colon \S^2\to\S^2$ such that conditions (i) and (iii) of Theorem~\ref{main:thm} hold \emph{contains a countable intersection of open and dense sets of maps of zero degree} in $H^{1/2}$ (or in some other topology). In spite of some efforts, we have not been able to settle that question.

The main novelty of Theorem~\ref{main:thm} and its proof is that (1) we show  that the singularities are unstable in a generic sense,
(2) in order to achieve that, we show how to combine an appropriately
modified idea of Hardt and Lin, applied by them only to
\emph{constant} boundary conditions $\phi\colon \S^2\to\{\ast\}$, with
a revisited version of  Almgren and Lieb's method of installing new
singular points, see \cite[Theorem~4.3]{AL}. A bridge between these
two ingredients is provided by a brief topological argument which
guarantees that for each boundary condition $\varphi$ with
$\deg\varphi=0$ there exist two antipodal points $\pm q\in \S^2$ such
that $\varphi$ maps them to the same point of $\S^2$,
$\varphi(q)=\varphi(-q)$. We select any pair of such points and,
roughly speaking, show how to insert many tiny bubbles into $\varphi$
close to those two antipodal points to obtain the new boundary
condition $\tilde\varphi$.  This way, $\varphi$ is changed only in two
little spherical caps centered at $\pm q\in\S^2$, so that the second
statement in (ii) in Theorem \ref{main:thm} does hold.

%The overall plan of the proof of Theorem~\ref{main:thm} is as follows. We first revisit Almgren and Lieb's method of installing new singular points, see \cite[Theorem~4.3]{AL}. Introducing appropriate modifications of their method, we select two antipodal points $\pm p\in \S^2$ such that $\varphi$ maps them to the same point of $\S^2$, $\varphi(p)=\varphi(-p)$ (the existence of such antipodal points is guaranteed by the assumption $\deg\varphi=0$), and insert many tiny bubbles into $\varphi$ close to those two points to obtain the new boundary condition $\tilde\varphi$. (In a sense, this is an imitation of the main idea of Hardt and Lin, applied by them to a \emph{constant} boundary condition $\phi\colon \S^2\to\{\ast\}$.) This way, $\varphi$ is changed only in two little spherical caps centered at $\pm p\in\S^2$, so that the second statement in (ii) in Theorem \ref{main:thm} does hold. 

To control the degree of $\tilde\varphi$ and to guarantee the uniqueness of minimizers of the Dirichlet integral in $W^{1,2}_{\widetilde\varphi}(\BB^3,\S^2)$, we employ the uniform boundary regularity of minimizing harmonic maps combined with the fact that harmonic maps are real analytic in the interior of the regular set. 

Finally, the distance from $\varphi$ to $\widetilde\varphi$ in $W^{1,p}$ is estimated by a technical, explicit analysis of the small bubbles. It is crucial here that $p<2$: the computations in Lemma~\ref{phi1phi2} break down for $p=2$, and an application of Almgren and Lieb's \cite[Theorem 2.12]{AL} shows that Theorem~\ref{main:thm} indeed fails for  $p=2$, see Remark~\ref{opt}. On the other hand, Hardt and Lin's Stability Theorem \cite{HL2} asserts that for a Lipschitz boundary mapping $\psi$ with unique energy minimizer $v$, each minimizer $u$ for a boundary mapping $\widetilde{\psi}$ sufficiently close to $\psi$ in the Lipschitz norm has the same number of singularities as $v$. In that sense, the $W^{1,p}$ topology for $1\le p<2$ in Theorem~\ref{main:thm} is optimal.

 \smallskip
The notation throughout the paper is standard. $B(x_0,r)=\{x\in\R^3:|x-x_0|<r\}$ is the standard Euclidean open ball. We write
\[
\partial E(\vp)=\int_{\S^2}|\nabla_T\vp|^2d\sigma
\]
to denote the \emph{boundary energy} of a map $\varphi\colon \S^2\to\S^2$. For a map $u\colon \BB^3\to \S^2$ we set
\[
J(x)\equiv J(u)(x) = \sqrt{\det \big(Du(x)Du(x)^T\big)}.
\]
If the rank of $Du(x)$ is maximal, i.\,e., equal to 2, then $J(u)(x)$ measures how $Du(x)\!\!\mid_V$, where $V$ is the orthogonal complement of $\ker Du(x)$, distorts the surface measure: for  an arbitrary ball $B$ centered at $x$, the Jacobian $J(u)(x)$ is equal to the ratio of $\h^2(Du(x) B)$ to $\h^2(B \cap V)$.

\section{Installing new singularities}

We start with a theorem of Almgren and Lieb, see \cite[Theorem  4.3]{AL}, which describes how to modify the boundary mapping so that its energy minimizer would have a singularity and the energy of the new minimizer would be almost the same as the energy of the initial one. This result will serve as a main tool in constructing $\w{\vp}$ in the proof of Theorem~\ref{main:thm}. 

Before giving the statement, we introduce the notation which will be useful in several places below.

\begin{defi}\label{singinstal}
	For a fixed map $\psi:\S^2\rightarrow\S^2$, which is smooth near a point $q\in\S^2$, and for a fixed number $\varrho>0$, we let $[\psi]_{q,\varrho}\colon\S^2\to\S^2$ denote any smooth boundary map which arises from $\psi$ by a small deformation in a neighborhood of $q$ so that the following four conditions are satisfied:
	\begin{itemize}
	  \item [(a)] $[\psi]_{q,\varrho}(x)=\psi(x)$ whenever $|x-q|\ge\varrho$;
	  \item [(b)] $[\psi]_{q,\varrho}(x)\equiv\psi(q)$ if $|x-q|=\varrho/2$;
	  \item [(c)] The restriction of $[\psi]_{q,\varrho}$ to the annular region $\frac\varrho 2<|x-q|<\varrho$ satisfies the Lipschitz condition with a Lipschitz constant $L$ which depends only on $\psi$ and \underline{not on} $\varrho$;
	  \item [(d)] $[\psi]_{q,\varrho}$ is a diffeomorphism of the spherical cap $\{|x-q|<\varrho/2\} \, \cap\, \S^2$ onto the punctured sphere $\S^2\setminus\{\psi(q)\}$ such that the boundary Dirichlet integral energy of $[\psi]_{q,\varrho}$ on this cap equals $8\pi+\text{o}(1)$ as $\varrho\rightarrow 0$.
	 \end{itemize}
\end{defi}

It is well known that such maps exist, e.g. a modification of the mapping obtained in \cite[Appendix A.2]{ABL}. If we identify the spherical cap from (d) with a disc and assume that $\psi(q)=(0,0,1)=N$ we can map a concentric smaller disc to the whole sphere without two spherical caps centered at $N$ and $-N$ consisting of points whose angular distance (in radians) from the point $N$ are smaller or equal $\frac\varrho 2$ and from the point $-N$ are greater or equal $\pi-\frac\varrho 2$, respectively. To do this we use a properly rescaled and rotated inverse stereographic projection. It is a smooth conformal mapping and therefore its Dirichlet energy is equal twice the Hausdorff measure of the image (and hence approaches $2\cdot4\pi$ as $\varrho\rightarrow0$). The remaining annuli from the domain can be mapped into the punctured spherical cap left in the image without changing the Dirichlet integral too much. For details see the proof of Lemma \ref{phi1phi2}.

We shall sometimes say that $[\psi]_{q,\varrho}$ arises from $\psi$ by inserting a smooth bubble at $q$. 

\begin{thm}\label{sing}
 Suppose $u:\BB^3\rightarrow\S^2$ is a minimizer which is unique for its boundary mapping $\psi:\S^2\rightarrow\S^2$ and which has an interior singularity at $p\in\BB^3$. Assume $\psi$ has finite boundary Dirichlet integral energy and is smooth near $q\in\S^2$ and let $\psi_j:\S^2\rightarrow\S^2$ be any sequence of continuous boundary mappings such that $\psi_j=[\psi]_{q,2/j}$ for all $j$ sufficiently large.

Finally, let $u_j$ be any minimizer in $\BB^3$ with boundary mapping $\psi_j$. Then, for all sufficiently large $j$, the mapping $u_j$ will have at least two interior singular points $q_j$ and $p_j$ such that $q_j\rightarrow q$ and $p_j\rightarrow p$ as $j\rightarrow\infty$.
\end{thm}

Since we had some trouble to follow the argument in \cite{AL} --- in particular the lines 11--14 on page 521 --- in full detail, we
include here a more detailed variant of Almgren and Lieb's proof, explaining the parts which were unclear for us.

\begin{proof}
The proof consists of five steps.

\smallskip
\noindent
\textsc{Step 1.} We first show that $u_j\rightarrow u$ strongly in $H^1$. By \cite[Theorem~1.1]{AL}, 
$$E(u_j)<C\sqrt{\partial E(\psi_j)}< C\sqrt{\partial E(\psi) + 8\pi + L},$$ 
so $\sup_j E(u_j)<\infty$ and $\sup_j \partial E(\psi_j)<\infty$. Therefore, by \cite[Theorem~1.2 part (4)]{AL}, after passing to subsequences $u_j$ and $\psi_j$ converge to some $u_0$ and $\psi_0$, respectively; the convergence (of each of these subsequences) is strong in $H^1$ and $u_0$ is a minimizer for its boundary mapping $\psi_0$. By the strong convergence we obtain the existence of another subsequence $j_k$ such that $\psi_{j_k}(x)\rightarrow\psi_0(x)$ for a.e.\ $x\in\S^2$. However, by its very definition $\psi_j(x)\rightarrow \psi(x)$ for all $x\in\S^2\setminus\{q\}$, so that $\psi_0=\psi$ a.e. and by the uniqueness of $u$ we obtain that $u_0=u$.

\noindent
\textsc{Step 2.} Now the existence of interior singular points $p_j$ of $u_j$ for sufficiently large $j$, as well as the convergence $p_j\rightarrow p$, follows from \cite[Theorem~1.8 part (2)]{AL}. (In short, if all $u_j$ were regular in a small neighborhood of $p$, the scaled energy of $u$ over a small ball $B(p,2/j)$ would be small enough to guarantee the regularity of $u$ at $p$.)

\smallskip
\noindent
\textsc{Step 3.} By the Boundary Regularity Theorem \cite{SU2} and monotonicity formula (see e.g. \cite[Corollary 1.7]{ABL}), we may choose an $R>0$ such that for each $r<R/2$ we have $\int_{B(q,2r)}|\nabla u(x)|^2\, dx<2\pi r$. 

\smallskip
\noindent
\textsc{Step 4.} As $\psi:\S^2\rightarrow\S^2$ is continuous near $q$, for any $\ve>0$ we may find a $\delta>0$ such that if $|x-q|<\delta$, $x\in\S^2$, then $|\psi(x)-\psi(q)|<\ve$. Let us fix $\ve>0$ and assume that for a~fixed small $r=\min(\delta,\frac{1}{2}R)$ independent of $j$   there is no singularity for each $u_j$ in the region $|x-q|<2r$.

Combining the elementary inequality $2|J(x)|\le|\nabla u(x)|^2$ and the co-area formula  
\[
\int_{\BB^3}|J(u)(x)|\, dx=\int_{w\in\S^2}\h^1(u^{-1}\{w\})\, d\h^2(w)\, ,
\]
see \cite[Chapter 3]{F}, one obtains
\begin{equation}\label{coarea}
 \int_{\BB^3}|\nabla u(x)|^2 \,dx\ge2\int_{w\in\S^2}\h^1(u^{-1}\{w\})\, d\h^2(w).
\end{equation}

For a fixed point $q\in\S^2$, to shorten the notation, we write $\dysk(q,a)=B(q,a)\cap\S^2$ for the spherical cap formed by the intersection of the ball $B(q,a)$ and the unit sphere. $\ann(q;a,b)=\bigl(B(q,b)\setminus B(q,a)\bigr)\cap\S^2$ is the  intersection of the annulus  $B(q,b)\setminus B(q,a)$ with the unit sphere. We also write $\uu_t=\partial\left(B(q,t)\cap\BB^3\right)$ for the boundary of the intersection of the unit ball and the ball centered at $q$ of radius $t$, and $\uu_t^-=\partial B(q,t)\cap \BB^3$ for the boundary of the ball centered at $q$ of radius $t$ intersected with the unit ball. Finally,  $\vv_\ve=B(\psi(q),\ve)\cap\S^2$ stands for the spherical cap established by the intersection of a ball centered at $\psi(q)$ of radius $\ve$ and a unit sphere.

We will use \eqref{coarea} to estimate the energy of $u_j$ for sufficiently large $j$'s in the region $r<|x-q|<2r$.  We consider $j>2/r$, so that the strict inclusion  $\dysk(q,2/j)\varsubsetneq\dysk_r:=\dysk(q,r)$ holds. By assumption $(d)$, $\psi_j\bigl(\dysk(q,1/j)\bigr)=\S^2\setminus\{\psi(q)\}$ and $\psi_j$ is injective in this small spherical cap, i.\,e., for any $y\in\S^2\setminus\{\psi(q)\}$ the set $\psi|_{\dysk(q,1/j)}^{-1}(y)$ consists of only one point. By $(a)$ and $(c)$, we also have $\psi_j\bigl(\ann(q;1/j,2/j)\bigr)\ \varsubsetneq \vv_\ve$  and $\psi_j\bigl(\ann(q;2/j,2r\bigr)\subseteq \vv_\ve$.
                   
\newcommand{\omitted}[1]{}
\omitted{
\begin{figure}[!ht]
\minipage{0.48\textwidth}
  \includegraphics[width=\linewidth]{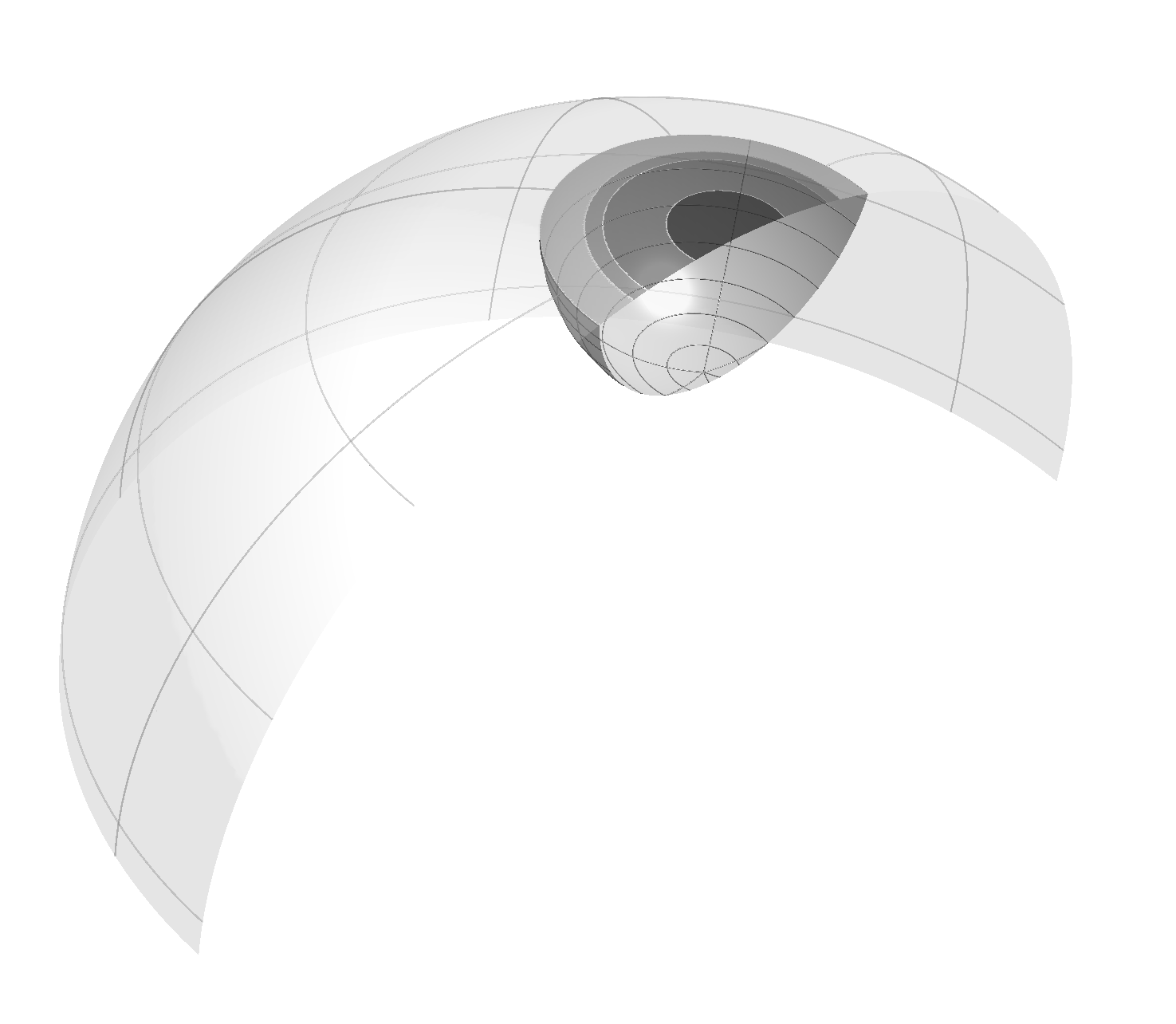}
  \caption{The domain of $u_j$. The green, red, blue and orange area are: $\dysk(q,1/j)$, $\ann(q;1/j,2/j)$, $\ann(q;2/j,r)$ and $\uu_t\setminus\dysk_r$ respectively, while the gray area illustrates $\BB^3\cap B(q,2r)$.} \label{fig:sferkidziedzina}
\endminipage\hfill
\minipage{0.48\textwidth}
  \includegraphics[width=\linewidth]{kolorowesferki.png}
  \caption{The image of $u_j(\uu_t)$. Colors illustrate the images of corresponding areas in Figure \ref{fig:sferkidziedzina}.
  }\label{fig:sferkiobraz}
\endminipage\hfill
\end{figure}
}  

\begin{figure}[!ht]   
{}\hfill{}
\includegraphics[height=7cm]{SferkiPociete-nowe.png}
{}\hfill{}  	
\includegraphics[height=6.5cm]{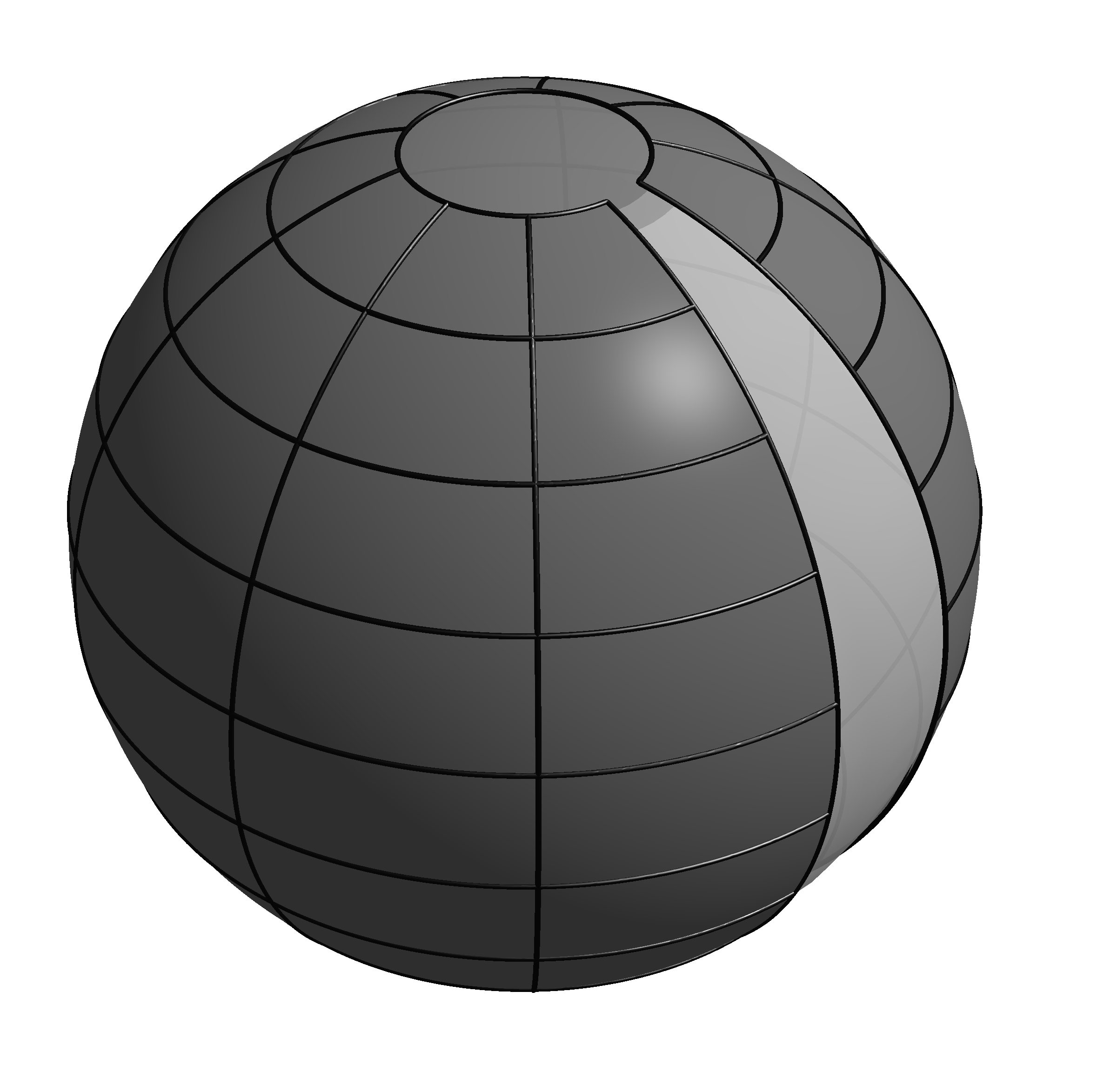}  
{}\hfill{}
\caption{Left: the domain of $u_j$. On the top part of $\S^2$, four shaded areas are visible: the dark gray cap $\dysk(q,1/j)$, a lighter annulus $\ann(q;1/j,2/j)$, still lighter narrow annulus $\ann(q;2/j,r)$, and the lightest $\uu_t\setminus\dysk_r$, with the rest of the boundary of $\BB^3\cap B(q,t)$ `hanging below'. Right: the image of $u_j(\uu_t)$, with corresponding shades of gray. The innermost dark cap $\dysk(q,1/j)$ is mapped to almost the whole sphere, like a blown-up piece of the bubble gum.}
\end{figure}
 
Since, by the assumption above, $u_j$  is continuous in the region $\{|x-q|<2r\}$, we have $\deg\big(u_j|_{\uu_t}\big)=0$ for every $t<2r$ because the set $\uu_t$ is topologically a sphere. Now, choose a number $t\in(r,2r)$, fix a point $y\in\S^2\setminus\{\psi\bigl(\ann(q;1/j,r)\bigr)\}$ and consider the set $(u_j|_{\uu_t})^{-1}(y)$ of all its preimages. We know that there exists precisely one point $a\in\dysk(q,1/j)$ such that $\psi_j(a)=u_j(a)=y$; since the degree is $0$ we deduce that there must be another point $b\in\uu_t^-$ such that $u_j(b)=y$ (with a reverse orientation than at $a$). This degree consideration shows that for each $t\in(r,2r)$ there exists a point $x_t\in \uu_t^-$ such that $u_j(x_t)=y$. Since $\S^2\setminus\{\psi\bigl(\ann(q;1/j,r)\bigr)\}\supset \S^2\setminus\vv_\ve$, we have $\h^1(u_j^{-1}\{w\})\ge r$ for all $w\in \S^2\setminus\vv_\ve$.

A simple computation yields $\h^2(\S^2\setminus\vv_\ve)=\pi\big(3+\brac{1-\frac{\ve}{2}}^2\big)$. 
Thus, for $\eps$ small,  by formula (\ref{coarea}) we obtain
\begin{align*}
 \int_{\{r<|x-q|<2r\}}|\nabla u_j(x)|^2 dx&\ge 2\int_{\S^2}\h^1(u_j^{-1}\{w\})d\h^2(w)\\
 &\ge 2\cdot r\cdot \pi\left(3+\left(1-\frac{\ve}{2}\right)^2\right)\\
 &> 7\pi r.
\end{align*}

Having in  mind the inequality $\int_{B(q,2r)}|\nabla u|^2dx<2\pi r$ from Step 3, this is a contradiction to the strong convergence obtained in Step 1. Thus in the region $|x-q|<2r$ for sufficiently large $j$'s each $u_j$ has a singularity $q_j$.

\smallskip
\noindent
\textsc{Step 5.} Now it suffices to show that $q_j\rightarrow q$ as $j\rightarrow \infty$. Since $\ve>0$ was arbitrary, we may choose a sequence of $\ve_j\searrow0$ such that the corresponding $r_j\searrow0$ and the regions $B(q,2r_j)$ in which the singularity $q_j$ appears will shrink to $\{q\}$.  
\end{proof}

\begin{rem}
 The assertion of Theorem \ref{sing} holds true if we replace each $\psi_j$ by a smooth approximation $\w{\psi}_j$ such that the modification in the region $|x-q|<\frac1j$ from Definition \ref{singinstal}~(d) remains a diffeomorphism of the smaller disc to the whole sphere without a small cap centered at $\psi(q)$, such that for sufficiently large $j$'s this cap is contained in $\vv_\ve$ from Step 4. One may easily check that it does not affect the proof.
\end{rem}

\section{Construction of $\w{\vp}$}

The main idea is as follows: we will modify $\vp$ on two antipodal sets (in fact, on two little antipodal spherical caps in $\S^2$) of small measures. The modified $\w{\vp}$ will be arbitrarily close to $\vp$ in the space $W^{1,p}$, $1\le p<2$ although its oscillations on these discs will be large in $C^0$. In the first step of the construction, we shall perturb the original mapping slightly, to make it constant on those two discs. Next, roughly speaking, we repeat the construction of Hardt and Lin in \cite{HL} in those regions to obtain our $\w{\vp}$. %\textcolor{red}{Moze napisac troche wyrazniej, ze sa tak naprawde trzy etapy konstrukcji, i napisac, po co jest kazdy z nich}

At the beginning of this section we recall without proofs a few known results which will be used in the proof of Theorem \ref{main:thm}. In the second part we construct our boundary condition and we close the section with the proof of Theorem \ref{main:thm}.

\subsection{Auxiliary propositions}

The following theorem is a restatement of boundary regularity criterion of Schoen and Uhlenbeck \cite{SU2}. This form, convenient for our purposes, is taken from \cite[Theorem 1.10 (2)]{AL}. 

\begin{thm}\label{UBC}
 There exists $\ve>0$ with the following properties. Suppose $f:\R^2\rightarrow\R$ is three times continuously differentiable with $f(0)=0$, $|\nabla f(0)|=0$, and each partial derivative of $f$ up to order 3 does not exceed $\ve^2$ in absolute value. Suppose also that $\vp_0:\R^2\rightarrow\S^2$ is three times continuously differentiable and that each partial derivative of $\vp_0$ up to order 3 does not exceed $\ve^2$ in absolute value. Finally suppose that $u^*$ is a minimizer in the region 
 \[
  \left\{(x,y,z):x^2+y^2\le 1 \textrm{ and } f(x,y)\le z\le 1\right\}
 \]
and the boundary mapping $\vp^*$ for $u^*$ satisfies the condition that
\[
 \vp^*(x,y,f(x,y))=\vp_0(x,y)  \textrm{ whenever } x^2+y^2<1.
\]
Then there is a two times continuously differentiable mapping $u_0:\R^3\rightarrow\S^2$ such that each partial derivative of $u_0$ up to order 2 does not exceed $\ve$ in absolute value and $u^*$ coincides with $u_0$ in the region 
\[
 \left\{(x,y,z):x^2+y^2\le\frac{1}{2} \textrm{ and } f(x,y)\le z\le\ve\right\}.
\]
\end{thm}

The next theorem was discovered by Almgren and Lieb; a precise statement can be found in \cite[Theorem~4.1 (1)]{AL}. It asserts that the boundary mappings having unique minimizers are dense in $H^1(\partial\BB^3)$. Theorem \ref{uniqueness}, and the trick used in its proof, will play an important role in our construction.

\begin{thm}\label{uniqueness}
 Suppose that $q$ is a point in $\partial\BB^3$, $\ve>0$, and that $\vp:\partial\BB^3\rightarrow\S^2$ is a boundary mapping with $\partial E(\vp)<\infty$. Then there is another mapping $\vp^*:\partial\BB^3\rightarrow\S^2$ which coincides with $\vp$ except possibly on that part of $\partial\BB^3$ within the ball $B(q,\ve)$, which differs from $\vp$ in $H^1(\partial\BB^3)$ norm by no more than $\ve$, and for which there is exactly one minimizer $u^*:\BB^3\rightarrow\S^2$ having boundary mapping $\vp^*$.
\end{thm}

The key observation in the proof of the above theorem is the following lemma which follows easily from the fact that harmonic maps into $\S^2$ are real analytic away from their singular points (see  the proof of  Theorem~4.1 in \cite{AL}). 
%\textcolor{red}{Czy piszemy, ze na brzegu $\Omega$ nie ma byc punktow osobliwych? Czy [AL] przypadkiem tego nie zakladaja?}

\begin{lem}\label{unilem}
 Suppose $\Omega$ is a proper subdomain of a larger domain $\Omega^*$ and $u$ is any minimizer (not necessary unique) in $\Omega^*$. Then the restriction $u|_{\Omega}$ of $u$ to $\Omega$ is the unique minimizer for its boundary mapping.
\end{lem}   

\subsection{Construction of $\w{\vp}$}
We start with the observation that if $\deg\vp=0$ then there exist two antipodal points $q,\, -q\in\S^2$ such that $\vp(q)=\vp(-q)$. For the existence of such $\pm q\in \S^2$, see for instance Granas and Dugundji \cite[Part~II, p.~94,~Thm.~(6.1)]{GD}. For the convenience of the reader, we give here the gist of a quick argument: assume on the contrary that $\vp(q)\not=\vp(-q)$ for all $q\in \S^2$; one then easily constructs a homotopy from $\vp$ to another map $\vp_0$ which preserves the antipodes, i.\,e., $\vp_0(q)=-\vp_0(-q)$ for each $q\in \S^2$. This is done as follows: for a given $q\in \S^2$, if we already have $\vp(q)=-\vp(-q)$ for some $q\in\S^2$, then the homotopy changes nothing; if $\vp(q)\not =-\vp(-q)$, then the two \emph{distinct} points $\vp(\pm q)\in \S^2$ determine a unique arc $\gamma$ of the great circle such that the length of $\gamma$ is smaller than $\pi$, and we let $\vp(\pm q)$ travel at equal, constant speeds towards two antipodal points $\pm \widetilde{q}$ on that great circle (note that $\gamma$ is located 
symmetrically on one of the half-circles joining $\pm \widetilde{q}$). However, it is well known that each map which preserves the antipodes must be of odd degree, a contradiction.

In the remaining part of this section, we simply say that $\pm q\in\S^2$ are the antipodal points of $\varphi$. First, we perturb $\varphi$ slightly by making it constant close to $\pm q$.

\begin{defi}\label{vp1}
For each $\vp\in C^\infty(\S^2,\S^2)$ with $\, \deg(\vp)=0$, having two antipodal points $\pm q\in \S^2$, and for a fixed number $\delta>0$ such that $\mathcal{H}^2\left(\vp(B( q,2\delta))\cup \vp(B(-q,2\delta))\right)<4\pi$, we let $\vp_1:\S^2\rightarrow\S^2$ denote any intermediate smooth mapping such that
\begin{itemize}
 \item [(1)] $\vp_1(x)\equiv\vp(q)$ for $x\in \S^2\cap\big( B(q,\delta)\cup B(-q,\delta)\big)$;
 \item [(2)] $\vp_1(x)=\vp(x)$ on $\S^2\setminus\bigl( B(q,2\delta)\cup B(-q,2\delta)\bigr)$;
 \item [(3)] On each of the two annuli $B(\pm q,2\delta)\backslash B(\pm q,\delta)$ the map $\varphi_1$ is given by a composition of $\varphi$ with a smooth diffeomorphism from the annulus to a punctured disc.
%to be specific, we take e.g. {$\vp_1(x)=\vp\left(\frac{2(|x\mp p|-\delta)\frac{x\mp p}{|x\mp p|}\pm p}{\abs{2(|x\mp p|-\delta)\frac{x\mp p}{|x\mp p|}\pm p}}\right)$} for $x\in B(\pm p,2\delta)\backslash B(\pm p,\delta)$.
\end{itemize}
\end{defi}

The parameter $\delta$ will be important in our further estimates. Therefore, we explain the choice of $\delta$ in the following lemma.

\begin{lem}\label{vp1estimation} For each $\eps>0$ there is a $\delta>0$ such that the map $\varphi_1$ specified in Definition~\ref{vp1} above has $\deg(\vp_1)=0$ and $\norm{\vp-\vp_1}_{H^{1}(\partial\BB^3)}<\frac{\ve}{4}$.
\end{lem}

\begin{proof}
By Sard's theorem (and the assumption that $\vp(B(q,2\delta))\cup \vp(B(-q,2\delta))$ is not of full measure) we may choose a regular value $y$ of $\vp_1$ such that $y\not\in\vp(B( 
\pm q,2\delta))$; by definition, the preimages of $y$ under $\varphi_1$ are the same as its preimages under $\vp$, so  that $\deg(\vp_1)=\deg(\vp)=0$. 
Since $\vp\in C^\infty(\S^2,\S^2)$, we have $\max_{x\in \S^2}|\nabla_T \vp(x)|<\infty$ and, as $\varphi\not\equiv\varphi_1$ only on the discs $B(\pm q, 2\delta)$ and $\nabla\varphi_1\equiv 0$ on $B(\pm q, \delta)$,
\begin{align}
%\begin{aligned}
 \frac 12 \norm{\vp-\vp_1}_{H^1(\partial\BB^3)}^2 \nonumber
& \le \int_{\S^2\cap B(q,2\delta)}|\nabla_T \vp(x)|^2 d\sigma + \int_{\S^2\cap(B(q,2\delta)\setminus B(q,\delta))}|\nabla_T \vp_1 (x)|^2\, d\sigma \\
&\quad{} + \int_{\S^2\cap B(-q,2\delta)}|\nabla_T \vp(x)|^2 d\sigma + \int_{\S^2\cap(B(-q,2\delta)\setminus B(-q,\delta))}|\nabla_T \vp_1 (x)|^2\, d\sigma \nonumber\\
& \quad{} + \frac 12 \h^2\big(\S^2\cap(B(q,2\delta)\cup B(-q,2\delta))\big)\cdot 2^2\label{normest2}\\
& \le 2(4\pi\delta^2 + 3\pi\delta^2)\, \max_{x\in\S^2}|\nabla_T\vp (x)|^2  + 16\pi\delta^2\nonumber\\ 
 &
\le 16\pi\delta^2\, \max_{x\in\S^2}(|\nabla_T \vp(x)|^2+1).\nonumber
%\end{align}
\end{align}
Thus choosing $\delta$ such that 
\[
 \delta<\frac{\ve}{4\sqrt{16\pi \max_{x\in\S^2} (|\nabla_T \vp(x)|^2+1)}}
\]
we obtain $\norm{\vp-\vp_1}_{H^1(\partial\BB^3)}<\frac{\ve}{4}$.

\end{proof}

We now fix $\varphi_1$ as above, and, perturbing it, define a new intermediate map $\varphi_2\colon \S^2\to \S^2$. Let $\alpha=4\arcsin \frac \delta 2$ denote the length  of the arc $\gamma \cap B(q,\delta)$, where $\gamma$ is any great circle through $q$. Without loss of generality suppose from now on that $q=(0,0,1)\in\S^2$. Roughly speaking, we are going to insert $2N$ appropriately small bubbles into $\varphi_1$, at points $\pm \xi_i$ close to $\pm q$, preserving the degree but forcing the minimizers to be singular at many points.

Recall also that $\dysk(a,2/j)\equiv B(a,2/j)\cap\S^2$ denotes a spherical cap centered at $a$.
\begin{defi}\label{phi2}
Let $\xi_i=\brac{0,\sin\brac{\frac{i\alpha}{N+1}}, \cos\brac{\frac{i\alpha}{N+1}}}\in B(q,\delta)$ for $i=1,\ldots,N$. For sufficiently large $j$'s, with $2/j \ll \delta/2N$, we define $\vp_2:\S^2\rightarrow\S^2$ as follows:
\begin{itemize}
 \item [(1)] $\vp_2(x)=[\vp_1]_{\xi_i,2/j}(x)$ for $x\in \dysk(\xi_i,2/j)$;
 \item [(2)] $\vp_2(x)=[\vp_1]_{\xi_i,2/j}(-x)$ for $x\in \dysk(-\xi_i,2/j)$;
 \item [(3)] $\vp_2\equiv \vp_1$ on $\S^2\setminus\left(\bigcup_{i=1}^N \dysk(\xi_i,2/j)\cup \dysk(-\xi_i,2/j)\right)$,
\end{itemize}
where $[\psi]_{a,b}$ is the modification of $\psi$ in the spherical cap $\dysk(a,b)$, see Definition \ref{singinstal}.
\end{defi}

Note that $\vp_2$ on each cap $\dysk(\xi_i,2/j)$ is either an orientation-preserving (degree $1$) or an orientation-reversing (degree $-1$) map onto $\S^2$, while on $\dysk(-\xi_i,2/j)$ it is of opposite orientation (respectively degree $-1$ or degree $1$) map onto $\S^2$. Since $\deg(\vp_1)=0$ we also have $\deg(\vp_2)=0$.

%\textcolor{red}{WYGLADZIC $\varphi_2$ w okolicy brzegow dyskow. Bez zmiany ogolnosci mozna to zrobic. Tw. o instalowaniu osobliwosci i tak gwarantuje, ze bedzie co najmniej $2N$ punktow osobliwych, bo szacowanie z co-area dziala.}

In the following lemma we will show that this procedure of inserting a single bubble to a map does not change the $W^{1,p}$ norm too much for $p<2$. 
\begin{lem}\label{phi1phi2}
Let $p<2$, then for each $N\in\N$, $\eps>0$ there is a (sufficiently large) $j$ such that $\norm{\nabla(\vp_1-[\vp_1]_{\xi_i,2/j})}_{L^p}<\frac{\eps}{8N}$ for each $i$.
\end{lem}
             
%\red{dowod w trakcie przerobek!!}

\begin{proof}
Fix $i\in\{1,\ldots,N\}$. Rotate $\S^2$ so that $\xi_i$ is the south pole $S=(0,0,-1)$, and work	in the spherical coordinates $(\phi,\theta)$, where $\phi\in[0,\pi]$ stands for the polar angle ($\phi=\pi$ corresponds to $S$)   and $\theta\in[0,2\pi]$ -- for the azimuthal angle.
	
Let $\Phi:=[\vp_1]_{S,2/j}$ on the spherical cap $\dysk(S,2/j)$. Without loss of generality we assume that $\Phi(S)=\vp_1(S)=(0,0,1)$. On the annulus $A(S;1/j,2/j)$ the map $\varphi_1$ is constant, and $\Phi$ is Lipschitz with a constant not depending on $j$. The main task is to estimate the $p$-energy of $\Phi$ on a smaller disk $D(S,1/j)$, blown by $\Phi$ onto a punctured sphere. For the sake of explicit estimates, we shall extract the formula for $\Phi$ on $\dysk(S,1/j)$.

In the spherical coordinates on $\S^2$ and the polar coordinates on the equatorial $\R^2$, the stereographic projection $\Phi_1$ is given by $\Phi_1(\phi,\theta) = (\cot\brac{\frac{\phi}{2}},\theta)$.  We have $D(S,1/j)=\{(\phi,\theta)\colon \gamma_j< \phi\le \pi\}$ with the latitude angle $\gamma_j=2\arccos \frac{1}{2j}$ on $\partial D(S,1/j)$. Thus, in the polar coordinates $(\rho,\vartheta)$ in $\R^2$,
\[
\Delta_j:= \Phi_1 (D(S,1/j))=\{(\rho,\vartheta)\colon 0\le \rho< d_j:=(4j^2-1)^{-1/2}\}\, .
\] 
We now set $\Phi\!\!\mid_{D(S,1/j)} = \Phi_2\circ \Phi_1\!\!\mid_{D(S,1/j)}$, where $\Phi_2$ sends an annulus $\tilde{A}_j\Subset\Delta_j\setminus\{0\}$  onto the whole $\S^2$ without two small caps (by rescaling $\tilde{A}_j$ and then applying $\Phi_1^{-1}$), and $\Phi_2$ is Lipschitz with an absolute constant on $\Delta_j\setminus\tilde{A}_j$ (so that the resulting map $\Phi$ satisfies all the requirements of Definition~\ref{singinstal}). 

Specifically, fix $0<\beta=\beta_j =  \frac{d_j}2$, and let $R=R_j=\cot \frac{\beta_j}2$ be the radius of the circle in $\R^2$ which is mapped to $\{\phi = \beta_j\}$ on $\S^2$ by the inverse stereographic projection $\Phi_1^{-1}$.  Set
\[
\lambda_j = \frac{R_j}{\beta_j}\, , 
\qquad r_j =\beta_j \tan^2\frac{\beta_j}2\, , 
\qquad\mbox{so that}\quad
\lambda_j r_j = \tan \frac{\beta_j}{2}=\cot \frac{\pi-\beta_j}{2}.
\] 
The circle $\partial B^2(0,\lambda_j r_j)\subset \R^2$ is mapped by the inverse stereographic projection to the latitude circle $\pi-\beta_j$ near the south pole in $\S^2$. Hence,
$
\tilde{A}_j:=\{(\rho,\vartheta)\colon r_j< \rho< \beta_j\}\subset \Delta_j
$
satisfies 
$
(\Phi_1^{-1} \circ \lambda_j\, \text{Id})\big(\tilde{A}_j\big) = \S^2\cap \{\beta_j< \phi< \pi-\beta_j\}$. We define the whole map $\Phi_2\colon \Delta_j \to \S^2$ by setting
\[
  \Phi_2(\rho,\vartheta) =
  \left\{ \begin{array}{ll}
     \brac{\pi-\rho \frac{\beta_j}{r_j}, \vartheta} & 
          \textrm{ for } 0\le \rho\le r_j,\\ [6pt]
     \brac{2\arccot\bigl(\frac{R_j}{\beta_j}\,\rho\bigr),\vartheta} & \textrm{ for } r_j< \rho < \beta_j \quad\mbox{(i.e., on $\tilde{A}_j$),}\\  [8pt]
     \brac{d_j-\rho, \vartheta} & \textrm{ for } \beta_j\le \rho< d_j.
  \end{array} \right.
\]
Finally, $\Phi=\Phi_2\circ \Phi_1$ on $D(S,1/j)$. 
%From now on, we fix $j$ and drop it from the notation. 
The key input to the $p$-energy of $\Phi$ comes from its behavior on  the annular region  $A^{(1)}_j := \Phi^{-1} \big(\S^2\cap \{\beta_j< \phi< \pi-\beta_j\}\big)$. A computation shows that  
\[
A^{(1)}_j = \ann\biggl(S; \frac{2r_j}{\sqrt{r_j^2+1}},\frac{2\beta_j}{\sqrt{\beta_j^2+1}}\biggr) \subset \S^2\, ;
\]
on this set, $\Phi$ is given by 
\[
\Phi(\phi,\theta) = 
\biggl( 
2\arccot\Bigl(\frac{R_j}{\beta_j}\cot \frac\phi 2 \Bigr),
\theta\biggr). 
\]
We estimate the $p$-energy of $\Phi$ on $A^{(1)}_j $ as follows, dropping the index $j$ next to $R,\, r,$ and $\beta$ for the sake of brevity:
 \begin{eqnarray*}     
  \lefteqn{
  \int_{A^{(1)}_j }|\nabla_T \Phi|^p d\sigma }\\
&=& \int_{2\arccot\beta}^{2\arccot r}\int_0^{2\pi} \brac{\brac{\frac{R\beta}{\beta^2\sin^2(\phi/2)+R^2\cos^2(\phi/2)}}^2 +\brac{\frac{1}{\sin\phi}}^2}^{p/2}\sin\phi\; d\theta \, d\phi \\
  &\le& 2\pi \int_{2\arccot\beta}^{2\arccot r} \brac{\brac{\frac{R\beta}{\beta^2+(R^2-\beta^2)\cos^2(\phi/2)}}^p+\sin^{-p}\phi}\sin\phi \; d\phi\; =:\; 2\pi \brac{I_1 + I_2}.
 \end{eqnarray*} 
\newcommand{\hide}[1]{}
\hide{
Changing the variables in a standard way, we obtain 
\[
 I_2=\frac12\brac{B_{(1-r^2)^2/(1+r^2)^2}\brac{\frac12,1-\frac{2}{p}}-B_{(1-\delta^2)^2/(1+\delta^2)^2}\brac{\frac12,1-\frac{2}{p}}},
\]
where $B_x(a,b)=\int_0^x s^{a-1}(1-s)^{b-1}\, ds$ is the incomplete beta function, see e.g. \cite[$\mathsection$8.17]{NIST:DLMF}. 
}        

Since $(ab)^{p/2} \le (a+b)^p$ for positive $a,b$, we have
\begin{align*}
 I_1 &\le \frac{2\cdot(R\beta)^p}{(\beta^2(R^2-\beta^2))^{p/2}} \int_{2\arccot\beta}^{2\arccot r} \frac{\sin\brac{\frac\phi2}}{\cos^{p-1}\brac{\frac\phi2}}d\phi\\
 &= \frac{4R^p}{(R^2-\beta^2)^{p/2}}
\int_{r/\sqrt{{r^2+1}}}^{\beta/\sqrt{{\beta^2+1}}}t^{1-p}\, dt\\
 &= \frac{4}{2-p}\frac{R^p}{(R^2-\beta^2)^{p/2}}
\brac{\brac{\frac{\beta^2}{\beta^2+1}}^{2-p}-\brac{\frac{r^2}{r^2+1}}^{2-p}}.
\end{align*}
Therefore, $I_1\rightarrow 0$ as $j\rightarrow\infty$, since $R=R_j\to \infty$ and $0< r=r_j<\beta=\beta_j\to 0$ (note that we use $1\le p<2$). Similarly,
$I_2\to 0$ as $j\to \infty$ by the absolute continuity of the integral, since $\int_0^\pi \sin^{1-p} \phi\, d\phi$ converges for $p<2$, and both endpoints of $I_2$ go to $\pi$ as $j\to \infty$.

Now to estimate the difference $\int_{\S^2}|\nabla_T(\vp_1-[\vp_1]_{S, 2/j})|^p d\sigma$ we note the following:
\begin{itemize}
 \item[(a)] $\vp_1$ and $[\vp_1]_{S,2/j}$ differ only on the spherical cap $\dysk(S,\frac2j)$ and $\vp_1$ is constant on that cap;
 \item[(b)] $\Phi$ is a composition of a conformal map ($\Phi_1$) and a Lipschitz map ($\Phi_2$) with constant $\beta_j/ r_j$ on the spherical cap $\dysk_{j}:=\dysk(S,{2r_j}/{\sqrt{r_j^2+1}})$;
\item[(c)] The $p$-energy of $\Phi$ on $A^{(1)}_j$ goes to zero as $j\to \infty$,
\item[(d)]$\Phi$ is a composition of a conformal map ($\Phi_1$) and a Lipschitz map ($\Phi_2$) with constant $1$ on the annular region $A^{(2)}_j:=\ann(S;{2\beta_j}/{\sqrt{\beta_j^2+1}},\frac1j)$;
\item[(e)] $\Phi=[\vp_1]_{S,2/j}$ is Lipschitz on $A^{(3)}_j=\ann(S;\frac1j,\frac2j)$ with constant $L$ dependent only on $\vp_1$ and not on $j$.
\end{itemize}
All this yields
\begin{align*}
 \int_{\S^2} |\nabla_T(\vp_1-[\vp_1]_{S, 2/j})|^p d\sigma 
&\le \biggl(\int_{D_j} + \int_{A^{(1)}_j} + \int_{A^{(2)}_j} +\int_{A^{(3)}_j}\biggr)\; |\nabla_T \Phi|^p d\sigma \\
 &\le 
\brac{\frac{\beta_j}{r_j}}^p \brac{2\h^2\brac{\Delta_j}}^{ p/2}\cdot\brac{\h^2\brac{D_j}}^{(2-p)/{2}} 
+ \int_{A^{(1)}_j}  |\nabla_T \Phi|^p d\sigma \\
 &\quad{}+ 
\brac{2\h^2\bracb{\Phi_1\bracb{A^{(2)}_j}}}^{p/2}
\cdot\h^2 \big(A^{(2)}_j\big)^{(2-p)/{2}} 
+ L^p\h^2\bracb{A^{(3)}_j}\,  
\end{align*}   
(we have applied H\"older's inequality and conformality of $\Phi_1$ to estimate the integrals over $D_j$ and $A^{(2)}_j$).
One easily sees that all the terms above converge to 0 as $j\rightarrow\infty$. Thus, we can choose $j$ such that 
\begin{equation}\label{norm12}
\int_{\S^2} |\nabla_T(\vp_1-[\vp_1]_{S, 2/j})|^p d\sigma < \brac{\frac{\eps}{8N}}^p\, , 
\end{equation}
%\qquad\mbox{and} \qquad\int_{\S^2} |\vp_1-[\vp_1]_{S,2/j}|^p d\sigma<\brac{\frac{\eps}{12 N}}^p,$$
thereby concluding the proof.
\end{proof}

\begin{rem}
 The map $[\vp_1]_{\xi_i,2/j}$ given in the proof of the above lemma by an exact formula is not smooth. In order to obtain a smooth map we choose an approximation $\widetilde{\Phi}$ of $\Phi$, so that $\widetilde{\Phi}\in C^{\infty}(\S^2,\S^2)$, the degree of $\Phi$ remains the same and the inequality \eqref{norm12} holds true.
 %We recall that by \cite[Theorem jakis tam czy Bethuel?]{BethuelZheng} the space $C^{\infty}(\S^2,\S^2)$ is dense in $W^{1,p}(\S^2,\S^2)$ for $1\le p \le2$.
\end{rem}

\begin{lem}\label{lemourphi} Fix $\delta_1>0$ sufficiently small.
One may modify $\vp_2:\S^2\rightarrow\S^2$ in a spherical cap of radius $\delta_1$, located away from all $\dysk(\xi_i,2/j)$, obtaining a new map $\vp_3\in C^\infty(\S^2,\S^2)$ such that $\norm{\vp_2-\vp_3}_{H^1(\partial\BB^3)}<10 \delta_1$ and $\deg(\vp_3)=0$, for which there is exactly one minimizer $\w{u}:\BB^3\rightarrow\S^2$ with $\w u\mid_{\partial \BB^3}= \vp_3$. 
\end{lem}

We essentially follow Almgren and Lieb's  proof of Theorem \ref{uniqueness}; the only important difference is that we have to make sure that our $\varphi_3$ is of degree 0. For the sake of completeness we state the argument in full.

\begin{proof}
{We extend  the ball $\BB^3$ slightly to obtain a new smooth domain $\Omega\varsupsetneq\BB^3$, so that $\Omega$ coincides with $\BB^3$ except in a ball $B(q^*,\delta_1)$, where $\delta_1<\frac{1}{4}\text{dist}(q,q^*)$ and, to fix the ideas, we choose $$q^*=\brac{0,-\sin\brac{\frac{\alpha}{2N}},\cos\brac{\frac{-\alpha}{2N}}} $$ 
away from all the $\xi_i$ and from the caps where the bubbles are inserted into $\varphi_1$. Roughly speaking, the new $\Omega$ is the union of $\BB^3$ and  of a  tiny and very flat bump of width $2\delta_1$ and height $\delta_1^5$, which is added close to $q^\ast$.  It is convenient to imagine $\partial \Omega$ as the graph of a smooth nonnegative function $\theta\colon \S^2\to [0,\infty)$ such that $\theta$ vanishes on $\S^2\setminus B(q^\ast,\delta_1)$ and close to $q^\ast$, after we flatten the sphere locally, \[\theta(\cdot)= \delta_1^5\eta\brac{\frac{\cdot}{\delta_1}}:\R^2\rightarrow [0,\infty),\]
where $\eta$ is a smooth nonnegative cutoff function supported in the unit disc with $\eta(0)>0$. Formally, we let $T:\S^2\setminus\{-q^*\}\rightarrow\R^2$ be a stereographic projection such that $T(q^*)=0$ and set
\[
 \Omega=\BB^3\cup\left\{y:T(\Pi(y))\in B(0,\delta_1)\subseteq\R^2 \text{ and } \text{dist}(y,\S^2)<\delta_1^5\eta\left(\frac{T(\Pi(y))}{\delta_1}\right)\right\}\, ,
\]
where $\Pi$ stands for the nearest point projection from $\partial\Omega$ to $\partial\BB^3$.
Multiplying $\eta$ by a positive constant, we may obviously assume that each partial derivative up to order 3 of $\delta_1^5\eta\brac{\frac{\cdot}{\delta_1}}$ does not exceed $\delta_1^2$ in absolute value. }

\begin{figure}[!t]
  \includegraphics[width=6cm]{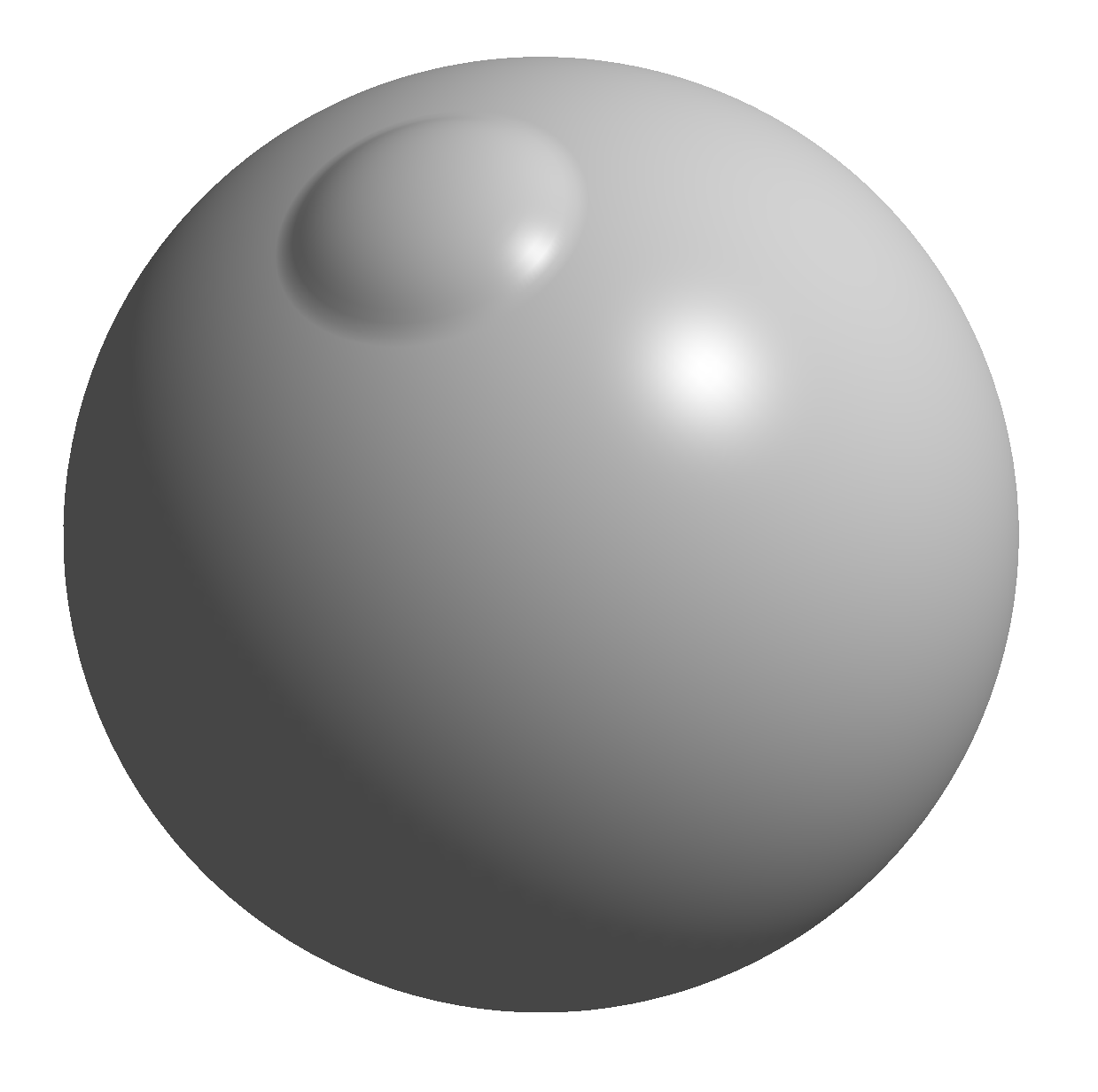}
  \caption{$\Omega$ is the union of a ball and a small flat bump.} \label{babel}
\end{figure}

Next we define a new mapping on the boundary of $\Omega$, $\vp^*:\partial\Omega\rightarrow\S^2$, by setting $\vp^*(x)=\vp_2(\Pi(x))$. By this definition we have $\vp^*\equiv\vp(q)$ on $B(q^*,4\delta_1)\cap \partial\Omega$. In particular, each partial derivative of $\vp^*$ is equal to 0 on that set and therefore does not exceed $\delta_1^2$ in absolute value. 

Let $u^*:\Omega\rightarrow\S^2$ be any minimizer for $\vp^*$. Then, $u^*|_{\BB^3}:\BB^3\rightarrow\S^2$ is the unique minimizer for its boundary mapping $\vp_3:=u^*|_{\partial\BB^3}$ by Lemma \ref{unilem}.
Note that by Theorem~\ref{UBC} $u^*$ is of class $C^2$ up to the boundary of $B(q^\ast,2\delta_1)\cap \Omega$.  This regularity assertion can easily be improved. To this end, we fix any smooth bounded domain $V\subset B(q^\ast,2\delta_1)\cap \Omega$ with, say, $V\supset \Omega\cap B(q^\ast,\frac 32 \delta_1)$, and with a $C^\infty$ boundary $\partial V\supset  B(q^\ast,\frac 32 \delta_1)\cap \partial\Omega$. An easy inductive argument using linear Schauder theory, see \cite[Thm.~6.19]{GT}, applied to $u^\ast\!\!\mid_V$ and the elliptic system $-\Delta u=|\nabla u|^2 u\equiv f$ on $V$, shows that in fact $u^\ast$ is of class $C^\infty (\overline V)$. Therefore $\varphi_3$ is of class $C^\infty(\S^2,\S^2)$.

 Next we show that $\deg(\vp_3)=0$. By the Uniform Boundary Regularity Theorem \ref{UBC}, the energy minimizer $u^*$ is two times continuously differentiable  at least on $B(q^*,\delta_1)$ and each of its partial derivatives does not exceed $\delta_1$, so that $|u^*(x)-u^*(y)|\le \sqrt{3}\, \delta_1 |x-y|$ by the mean value theorem. Thus, if $x\in\partial\Omega\cap B(q^*,\delta_1)$ and $y\in\S^2\cap B(q^*,\delta_1)$, then \[|\vp(q)-\vp_3(y)|=|u^*(x)-u^*(y)|\le \sqrt{3}\, \delta_1|x-y|< 4\delta_1^2\, .\]
To compute the degree of $\vp_3$, choose any regular value of $\vp_3$ away from $\S^2\cap B\bigl(\vp(q),4\delta_1^2\bigr)$. Its preimages under $\vp_3$ will be the same as those under $\vp_2$. Thus, the degree of $\vp_3$ must be the same as that of $\vp_2$, i.\,e., equal to zero. 

Finally, since by Theorem \ref{UBC} each partial derivative of $u^*\mid_{\partial\BB^3}=\varphi_3$ does not exceed $\delta_1$ on  $B(q^*,\delta_1)$, and on the set $\{\vp_2\ne\vp_3\}$ the mapping $\vp_2$ is constant, we have the estimate 
\begin{align*}
 \norm{\vp_2-\vp_3}_{H^1(\partial\BB^3)}^2 & = \int_{\{\vp_2\ne\vp_3\}}\Big(\left|\nabla_T \vp_2- \nabla_T\vp_3\right|^2 +
\left|\vp_2- \vp_3\right|^2\Big)\, d\sigma\\
& \le  2 \int_{\{\vp_2\ne\vp_3\}} \left|\nabla_T\vp_3\right|^2\, d\sigma +2^2 \h^2(\{\vp_2\ne\vp_3\}) \\ & 
%<  2\int_{B(p^*,\delta_1)}\left|\nabla u^*\right|^2 dx +4\pi \delta_1^2 
< 10\pi\delta_1^2\qquad\mbox{for $\delta_1<1$.}
 \end{align*}
Therefore, for sufficiently small $\delta_1$ we conclude that
$\norm{\vp_2-\vp_3}_{H^1(\partial\BB^3)}<10\delta_1$.
\end{proof}

\begin{proof}[Proof of Theorem \ref{main:thm}] It remains to check that the mapping $\vp_3$ given in Lemma \ref{lemourphi} has the properties (i)--(iii). 
	%and then to finish the proof we will choose an approximation $\w\vp$ of $\vp_3$ so that $\w\vp\in C^\infty(\S^2,\S^2)$.

\smallskip
\noindent
(i) and (iii): By Lemma \ref{lemourphi}, a minimizer $u_3$ for the boundary condition $\vp_3$ is unique and of degree 0. The proof that $u_3$ has at least $2N$ singularities is essentially the same as in Theorem \ref{sing}, therefore we skip it. 

 \smallskip
 \noindent
 (ii): Fix $\ve>0$. We now attune $\delta, \, \delta_1$ and $j$ to obtain $\norm{\vp-\vp_3}_{W^{1,p}}<\ve$. We first choose $\delta>0$ as in the proof of Lemma \ref{vp1estimation}, then $j$ as in the proof of Lemma \ref{phi1phi2}. Next, we fix $\delta_1<\frac{1}{10}$. Finally, we recall that $\vp$ differs from $\vp_3$ only on the two spherical caps $\S^2\cap B(\pm q,2\delta)$ whose $\h^2$ measure is $8\pi\delta^2$, shrinking $\delta$ if necessary we obtain $\h^2(\{x\in\S^2:\vp(x)\ne\vp_3(x)\})<\frac{\ve}{4}$ and hence
\begin{align*}
 \norm{\vp-\vp_3}_{W^{1,p}}&<\norm{\vp-\vp_3}_{L^p}+\norm{\nabla(\vp-\vp_1)}_{L^2}\cdot\brac{\h^2(\{\vp\ne\vp_1\})}^{\frac{2-p}{2p}}\\
 &\quad+\norm{\nabla(\vp_1-\vp_2)}_{L^p}+\norm{\nabla(\vp_2-\vp_3)}_{L^2}\cdot\brac{\h^2(\{\vp_2\ne\vp_3\})}^{\frac{2-p}{2p}}\\
 &< \frac{\ve}{4} + \frac\ve4\cdot{\brac{\frac\ve4}}^{\frac{2-p}{2p}} + 2N\cdot\norm{\nabla(\vp_1-[\vp_1]_{\xi_i,2/j})}_{L^p} + 10\delta_1\cdot \frac{\ve}{4}<  \ve.
 \end{align*}
\end{proof}

\begin{rem}\label{opt}
 Theorem \ref{main:thm} does not hold if we replace the norm $W^{1,p}(\S^2, \S^2)$ for $1\le p<2$ by $W^{1,2}(\S^2,\S^2)$. 
\end{rem}

\begin{proof}
 Let $\psi:\S^2\rightarrow\S^2$ be a constant map. For this boundary condition of degree 0 there exists exactly one minimizer $u:\BB^3\rightarrow\S^2$, $u\equiv const$ for which the Lavrentiev gap phenomenon does not hold. If we modify the boundary map into $\widetilde{\psi}$, without changing its degree, so that the gap phenomenon would hold we would have to install at least 2 singular points into each minimizer. By \cite[Theorem 2.12]{AL} the number of singularities is bounded by a universal constant $C_{AL}$ times the boundary energy, which in our case would give $\int_{\S^2}|\nabla_T\widetilde{\psi}|^2d\sigma\ge\frac{2}{C_{AL}}$. Therefore the modified boundary mapping $\widetilde{\psi}$ with singularities for each corresponding minimizer cannot be arbitrary close to the constant map $\psi$ in the $W^{1,2}$ norm.
\end{proof}

\section{A remark on the nonuniqueness in the class of minimizing mappings}

In the following we explain how the boundary mapping constructed in Theorem~\ref{main:thm} leads to a nonuniqueness example, similar (in the construction) to that of \cite[Section 5]{HL}.

\begin{rem}
Fix any $M\in \N$. There exist a mapping $\varphi_\tau\in C^\infty(\S^2,\S^2)$, $\deg(\varphi_\tau)=0$, which serves as a boundary data for at least two energy minimizing maps from $\BB^3$ to $\S^2$ having different number of singularities (one of them at most $M$; the other one at least $M+2$). 
\end{rem}

%\begin{proof}
 Indeed, let $\psi\in C^\infty(\S^2,\S^2)$ be any mapping having exactly $M\in\N$ singular points such that $\deg(\psi)=0$ and for which there exists unique energy minimizer $w\in W^{1,2}(\BB^3,\S^2)$. We construct $\w{\psi}\in C^\infty(\S^2,\S^2)$ as in Theorem \ref{main:thm} for which $\deg(\w{\psi})=0$ and there exists precisely one energy minimizing mapping $\w{w}\in W^{1,2}(\BB^3,\S^2)$ with at least $M+2$ singularities.
 
 Since the mappings $\psi$ and $\w\psi$ are homotopic, there exist a smooth family of smooth mappings $\{\varphi_t\}_{t\in[0,1]}$ such that $\varphi_0=\psi$ and $\varphi_1=\w{\psi}$.
 
 From the Stability Theorem obtained in \cite{HL2} we deduce that for $t$ sufficiently close to $0$ each energy minimizer with boundary data $\varphi_t$ has exactly $M$ singular points. Let
\begin{multline*}
 \tau=\sup\left\{t\in [0,1]: \text{ each energy minimizer with boundary data } \varphi_t \right.\\
\left.\mbox{ has at most } M \text{ singular points in $\BB^3$}\right\}.
 \end{multline*}
Then, $0<\tau<1$. We may choose a sequence $s_i\nearrow\tau$ and a sequence of energy minimizing maps $u_i\in W^{1,2}(\BB^3,\S^2)$ having at most $M$ singular points such that $u_i\!\mid_{\S^2}=\varphi_{s_i}$. Similarly we choose $t_i\searrow\tau$ with a sequence of minimizing mappings $v_i\in W^{1,2}(\BB^3,\S^2)$ having at least $M+2$ singularities,
%(possibly all $t_i=\tau$), 
$v_i\!\mid_{\S^2}=\varphi_{t_i}$. (Since we consider boundary maps of degree zero, and it is known that the degree of a minimizing harmonic map on a small sphere around a singular point is $\pm 1$, the number of singular points must jump at least by 2.) Passing to subsequences without changing notation, we obtain $u_i\rightarrow u$ and $v_i\rightarrow v$, the convergence is strong in $W^{1,2}$ and $u\!\mid_{\S^2}=\varphi_\tau=v\!\mid_{\S^2}$.

The mapping $u$ has at most $M$ singularities. (It is plausible that one might prove that the number of singularities equals $M$, by choosing the homotopy appropriately.) Indeed, assume $u$ has at least $M+2$ singular points. Then, by \cite[Theorem 1.8 (2)]{AL}, in an arbitrarily small ball around each singularity of $u$ there would be a singularity of $u_i$ for $i$ sufficiently large, a contradiction. 

On the other hand, $v$ has at least $M+2$ point singularities. Recall that each $v_i$ has at least $M+2$ singularities and again by \cite[Theorem 1.8]{AL} we know that singular points converge to singular points. To see that $v$ has at least $M+2$ singularities we must exclude the possibility that some singularities of the $v_i$'s come together and cancel. By \cite[Theorem 2.1]{AL} there exists a universal constant $C$ such that if $d$ denotes the distance from a singularity $a$ to the boundary of the ball then there is no other singularity within distance $Cd$ from $a$. Thus, the singularities of $v_{i}$ cannot merge in the interior of $\BB^3$. Moreover,  by Theorem \ref{UBC},  there is a neighborhood of the boundary which contains no singularities of $v$ and of the $v_i$'s sufficiently close to $v$ (as the $\vp_{t_i}$'s and $\vp_\tau$ are close to each other in $C^\infty$). This precludes the case of singularities merging in the limit at the boundary.    

\subsection*{Acknowledgement.} The work of both authors has been partially supported by the NCN grant no. 2012/07/B/ST1/03366.

\bibliographystyle{plain}
\bibliography{bibliografia}

\bigskip
\noindent
{\sc Katarzyna Mazowiecka}\\
Instytut Matematyki\\
Uniwersytet Warszawski\\
ul. Banacha 2\\
PL-02-097 Warsaw \\
POLAND\\
E-mail: {\tt K.Mazowiecka@mimuw.edu.pl}

\vspace{.5cm}

\noindent
{\sc Pawe\l{} Strzelecki}\\
Instytut Matematyki\\
Uniwersytet Warszawski\\
ul. Banacha 2\\
PL-02-097 Warsaw \\
POLAND\\
E-mail: {\tt pawelst@mimuw.edu.pl}
\end{document}